\documentclass{article}
\usepackage[latin1]{inputenc}
\usepackage{amsmath,amssymb}
\usepackage{amsfonts}
\usepackage{booktabs}
\usepackage{lscape}
\usepackage{color}
\usepackage{nicefrac} 
\definecolor{Gray}{gray}{0.9}

\title{Shadow wave solutions for a scalar two-flux conservation law with Rankine-Hugoniot deficit}

\newtheorem{thm}{Theorem}[section]
\newtheorem{df}{Definition}[section]

\newtheorem{ex}{Example}[section]

\usepackage{graphicx}

\newcommand{\defeq}{\mathrel{:\mkern-0.25mu=}}

\usepackage{multirow}

\title{Shadow wave solutions for a scalar two-flux conservation law with Rankine-Hugoniot deficit}

\author{Tanja  Kruni\'{c}, Marko Nedeljkov}
\date{}

\begin{document}
\maketitle

\begin{abstract}
The paper deals with scalar conservation laws
having a flux discontinuity at $x=0$ without a weak solution that
satisfies the classical Rankine--Hugoniot jump condition at $x=0$. 
We are using unbounded solutions in the form of shadow waves 
supported by the origin for solving that problem. 
The shadow waves are nets of piecewise constant
functions approximating a shock wave with added a delta function
and sometimes another unbounded part. 

{\it keywords}: {conservation laws with discontinuous flux functions, shadow waves; delta shock waves, singular shock waves}
\end{abstract}

\section{Introduction}	

Suppose that a following scalar conservation law 
\begin{equation}
\label{mishrinaJednacina}
\begin{split}
\partial_{t} u+\partial_{x}f(x,u)&=0, 
 \end{split}
\end{equation}
with a discontinuous flux function (``two flux functions'')
\begin{equation}
\label{fluks}
f(x,u):=H(x)f_r(u)+(1-H(x))f_l(u)
\end{equation}
is given. Here,
$u=u(x,t)\in \Omega\subset\mathbb{R},\;(x,t)\in \mathbb{R}\times \mathbb{R}_{+}$, $f_l,f_r\in C^{1}(\Omega)$ and $H(x)$ is the Heaviside step
function. The functions $f_l$ and $f_r$ represent the left and right-hand
sides of the flux jump. They are called ``left'' and ``right flux'' 
for short. The aim of this paper is to solve the cases when there is 
no a weak solution consisting of classical waves (shock or rarefaction waves)
to the Riemann problem \eqref{fluks} with the initial data
\begin{equation}
u(x,0)=\left\{\begin{array}{ll}\label{kx}
	u_{0},& x<0\\
	u_{1},& x>0
	\end{array}\right..
\end{equation} 
For example, such a situation occurs when either $f_l$ or $f_r$
is of convex type, and the other of concave type and the graph of one of the
parts of the flux jump lies below the other part (i.e. $f_l(u)>$ or ($<$)
$f_r(u)$) on the entire domain $\Omega$, see Fig. \ref{noRH}. Convex (concave)
type function present a generalization of a convex (concave) function, 
the precise definition is given in Section \ref{ini}. 
The reason why there is no solution is that
the  Rankine-Huginiot jump condition between two sides of the flux
\begin{equation}\label{rh}f_l(u^-)=f_r(u^+),\end{equation}
with
\begin{equation}
\label{uPlus}
u^-=\lim_{x\rightarrow 0^-}u(x,t)\quad\text{and}\quad u^+=\lim_{x\rightarrow 0^+}u(x,t),
\end{equation}
cannot be satisfied. 

Our construction is based on solving the so-called overcompressive case first, i.e. the case when 
\begin{equation}\label{over}
f_{l}^{\prime}(u_0)\geq 0\geq f_{r}^{\prime}(u_1).
\end{equation}

\begin{figure}[t]
\begin{center}
\begin{tabular}[ht]{@{}cc@{}}\small{A)$f_l$ -- convex type, $f_r$ -- concave type}&\small{B)$f_l$ -- convex type, $f_r$ -- convex type}\\
\includegraphics[height=32mm]{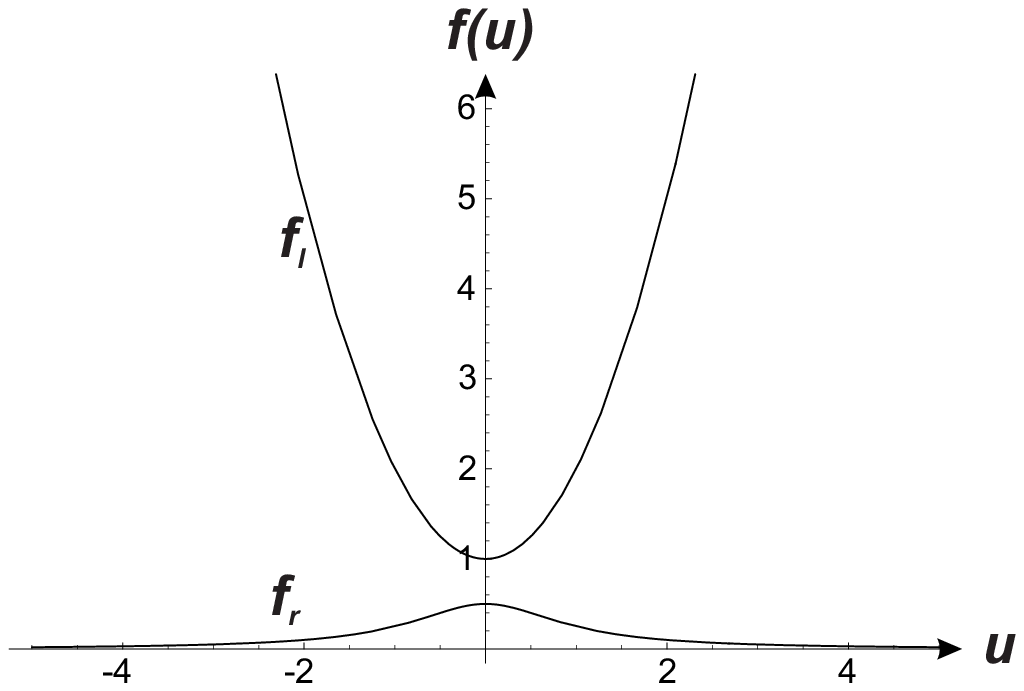} &\includegraphics[height=32mm]{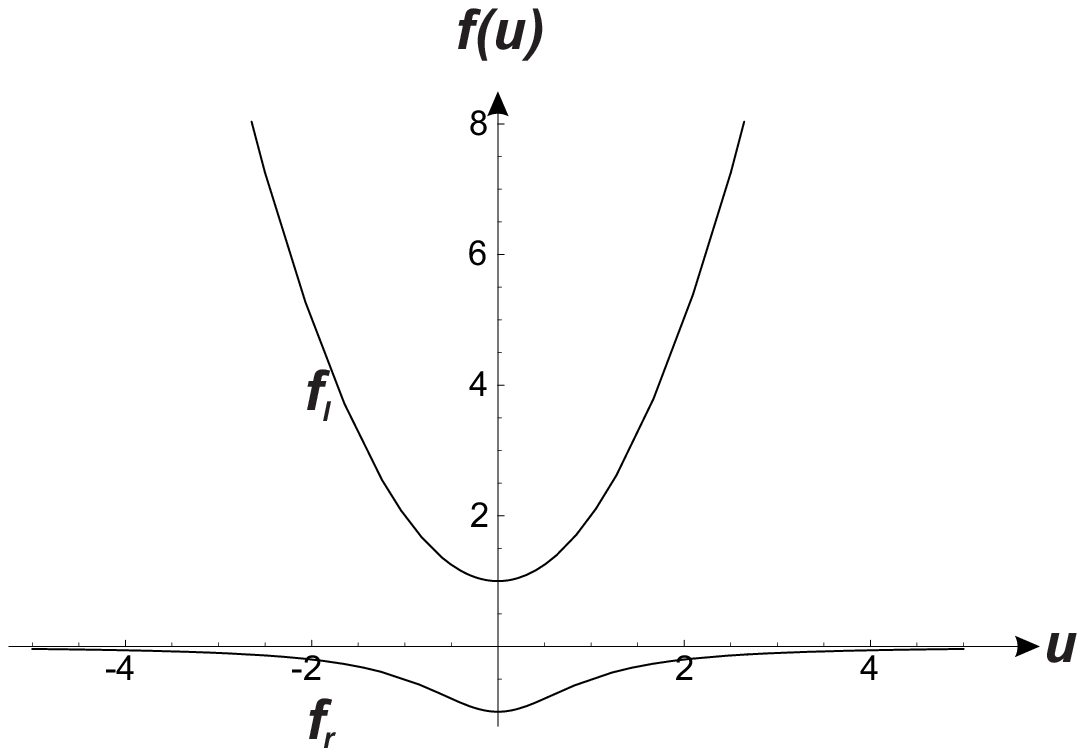}  \\
\tiny{$f_l=u^2+1,f_r=\dfrac{1}{2(u^2+1)}$ }& \tiny{$f_l=u^2+1,f_r=\dfrac{-1}{u^2+1}$ } \\
\small{C)$f_l$ -- concave type, $f_r$ -- convex type}&\small{D)$f_l$ -- concave type, $f_r$ -- concave type}\\
\includegraphics[height=32mm]{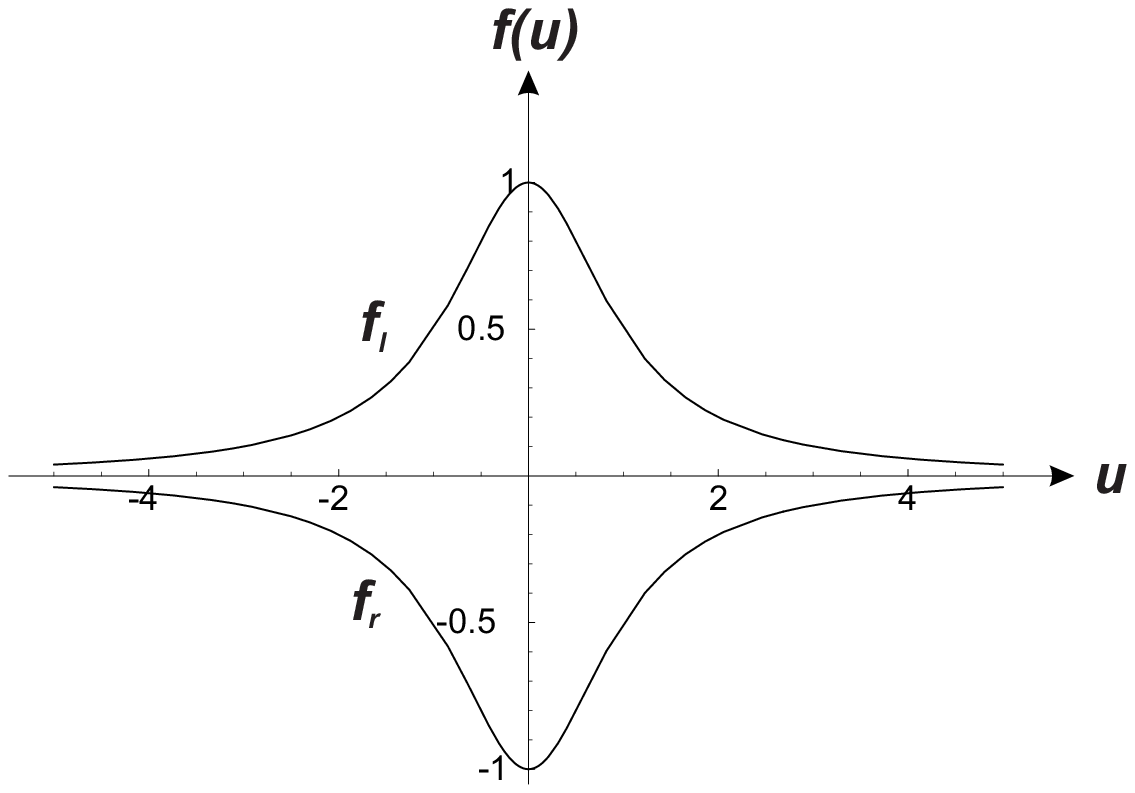} &\includegraphics[height=32mm]{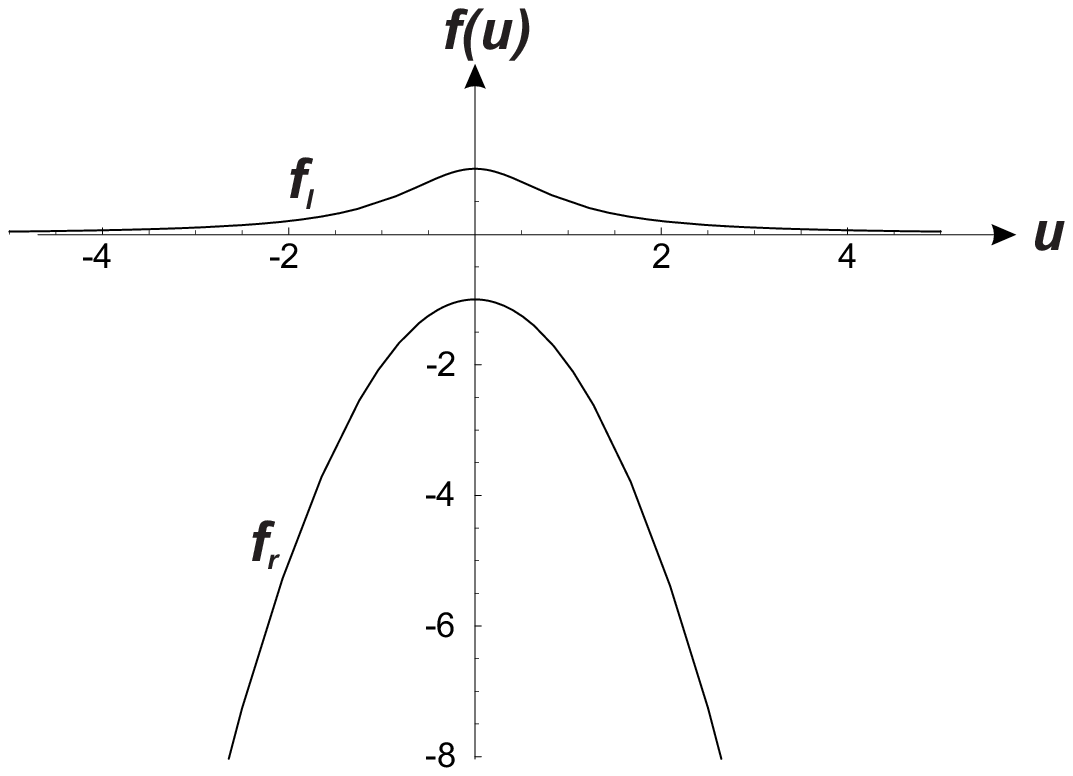}  \\
\tiny{$f_l=\dfrac{1}{u^2+1},f_r=-\dfrac{1}{u^2+1}$ }& \tiny{$f_l=-(u^2+1),f_r=\dfrac{1}{u^2+1}$ }
\end{tabular}
\end{center}
\caption{Examples of flux geometries with no classical solutions}
\label{noRH}
\end{figure}


Classical solutions of conservation laws with discontinuous flux functions 
have been widely investigated in the recent years.
Equations of the form \eqref{mishrinaJednacina} are considered by many authors 
under the assumption that both fluxes $f_{l},\;f_{r}$ are of convex (concave) 
type. In these cases, equation \eqref{mishrinaJednacina} has weak solutions 
consisting of a left-going (backward) and a right-going (forward) shock or 
rarefaction wave connected with a stationary shock wave which satisfies 
\eqref{rh} at $x=0$. There are different solution concepts and 
admissibility criteria for solutions. Some of them are the ``interface entropy 
condition'' (\cite{MishraRad}), the ``crossing condition'' (\cite{clari2}) and 
the ``existence of admissible discrete shock profiles'' condition 
(\cite{nase}).  Admissibility conditions 
for a general flux shapes are given in paper \cite{AC}. 
Let us note that in all these cases the primary assumption
is that the Rankine--Hugoniot condition is satisfied at the discontinuity point.
In \cite{signCh}, the interface admissibility conditions 
are extended to so-called ``convex-concave" fluxes  when one of them, $f_{l}$ 
or $f_{r}$ is of convex type, and the other one is of concave type. 
The authors look at the case when $f_{l}$, $f_{r}$ cross each other in 
some points $X$ and $Y$, and take initial data from the interval $[X,Y]$. 
The existence of solutions consisting of a backward and a forward shock or 
rarefaction wave connected with a 
stationary shock at $x=0$ is proved in this case.  

In \cite{general}, equation \eqref{mishrinaJednacina} is considered in the
context of certain flux geometries and sets of initial data for which the
Rankine-Hugoniot condition at $x=0$ cannot be satisfied, i.e.
the above conditions are not met. For example, 
when $f_l$ is of convex type and $f_r$ of concave type and
$f_l(u)>f_r(u)$ for all $u\in\Omega$, or when both fluxes are
affine and $f_{l}^{\prime}>0>f_r^{\prime}$.  The authors replace the
Rankine-Hugoniot condition with a relaxed ``generalized Rankine-Hugoniot 
condition'' and accordingly define bounded ``generalized entropy solutions'' 
(GES) which can be obtained as limits of a modified Godunov type finite 
difference scheme. In the aforementioned cases, the generalized entropy 
condition reads as $f_l(u^-)\geq f_r(u^+)$  and  GES just reflects the initial 
value discontinuity provided it satisfies the overcompressibility condition 
\eqref{over}. Relaxed solutions obtained in such a way do not satisfy equation 
\eqref{mishrinaJednacina} in the distributional sense. These results 
are compared with the ones obtained by using the staggered mesh Godunov type 
scheme in \cite{Karlsen}. They stated that in the observed cases this scheme 
gives a ``big overshoot'' at $x=0$, but coincides with GES apart from the 
``overshoot''. We try to give another explanation of that phenomena. 
The Rankine--Hugoniot deficit (when the condition can not be satisfied)
is connected with non--bounded generalized solutions containing the Dirac
delta function (delta, singular shocks, etc.). We use shadow waves
(\cite{Marko}) in order to obtain a solution. 
The shadow wave solution concept relies on representing 
singular solutions by nets of piecewise constant functions. The basic idea in 
constructing stationary shadow waves at $x=0$ is the perturbation of the speed 
$c=0$ of a wave with the states $u_{0}, u_{1}$ from the left and right side of 
the wave by a small parameter $\varepsilon$. After the perturbation, the 
states $u_{0}, u_{1}$ of the original wave are connected by three shocks: 
The first one connects $u_{0}$ and $u_{0,\varepsilon}$, the second 
$u_{0,\varepsilon}$ with $u_{1,\varepsilon}$, while $u_{1,\varepsilon}$ and 
$u_{1}$ are connected by the third one. The middle wave has zero speed for 
$\varepsilon<<1$, whereas the others have a speed of order $\varepsilon$.
The two intermediate states $u_{0,\varepsilon}$ and $u_{1,\varepsilon}$ tend 
to infinity when $\varepsilon\rightarrow 0$, and the solution tends to a 
singular shock wave or a delta shock wave (see Definition \ref{minor2} below). 
All calculations are made in the sense of distributions by using 
Rankine-Hugoniot conditions, i.e. the equality is taken to be the equality of 
the distributional limit, (see Section \ref{ini}). For comparison, one can see 
Fig. \ref{ilustracije1} as a representation of shadow waves and GES solutions 
obtained by numerical procedures. 
A precise definition of a shadow wave is given in 
Section \ref{ini}. The Rankine-Hugoniot deficit in this case is 
\begin{equation}\label{deficit}
\kappa\defeq f_l(u_0)-f_{r}(u_1),
\end{equation} 
\begin{figure}[t]
\begin{center}
\begin{tabular}[ht]{@{}cc@{}}
\includegraphics[height=45mm]{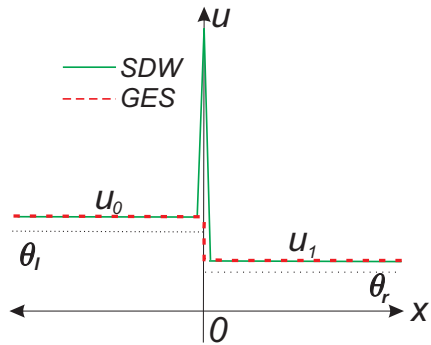} 
\end{tabular}
\end{center}
\caption{Shadow waves vs. GES}
\label{ilustracije1}
\end{figure}
and there is a solution in the form of a stationary shadow 
wave for nonzero Rankine-Hugoniot  deficit at $x=0$. 

Relation \eqref{over} is taken to be the admissibility criterion which means 
that all the characteristics run into the shock wave from both sides. It is 
called the overcompressibility condition in the sequel. Later in the paper, we 
consider a more general set of initial data \begin{equation}
\label{kxOpsti}
u(x,0)=\left\{\begin{array}{ll}
	u_{l},& x<0\\
	u_{r},& x>0
	\end{array}\right.
\end{equation}
which may not satisfy condition \eqref{over}. In this case, the states $u_l$ 
and $u_r$ have to be connected first by classical waves with some states $u_0$ 
and $u_1$ respectively, which satisfy condition \eqref{over}. The speed of the
wave connecting $u_l$ and $u_0$ has to be non-positive, whereas the speed of 
the wave connecting $u_1$ and $u_r$ has to be non-negative, of course. That 
condition provides a unique choice of the states $u_0$ and $u_1$, as one will 
see in Section \ref{riman}. 

This paper is organized in the following way. Suppose that $f_l$ is of 
convex type, $f_r$ of concave type and $f_l(u)>f_r(u)$ for all $u$. After some 
introductory facts and giving assumptions on the flux geometry in Section 
\ref{ini}, we prove the existence of shadow waves for 
\eqref{mishrinaJednacina}, \eqref{kx}, \eqref{over} in the overcompressive  case, in Section \ref{egzistencija}. In Section \ref{riman}, the results are 
extended to the general Riemann problem. Then we consider our results to 
equation 
\eqref{mishrinaJednacina} in the case of linear fluxes in Section 
\ref{prosirenje}.  Finally, in Section \ref{numerika111}, Eq. 
\eqref{mishrinaJednacina} is rewritten in the form of a system of conservation 
laws and solved numerically using Godunov's scheme. It is an experimental 
verification that the obtained unbounded solution agrees with the 
theoretically expected shadow wave, obtained by using the numerical test from 
\cite{natasa}. \\

\subsection{Preliminaries}\label{ini}
Let us fix the notation and assumptions used in the paper.
As already mentioned in the introduction, we are dealing with flux geometries 
for which the Rankine-Hugoniot condition cannot be satisfied. For simplicity, 
we will consider only the case when $f_l(u)>f_{r}(u)$ and the fluxes are of 
convex or concave type. These two function types are described  in the 
following definition.
\begin{df} \label{tip}
Let $f\in C^{1}(\Omega)$, $\Omega\subset \mathbb{R}$. Then $f$ is said to be

\noindent 
a convex type flux if it has one minimum and no maximum in $\Omega$,

\noindent
a concave type flux if it has one maximum and no minimum in $\Omega$.
\end{df}

The following four cases are possible, as shown in Fig. \ref{noRH}:

\noindent
Case A: $f_{l}$ is of convex type and $f_{r}$ is of concave type,

\noindent
Case B: both $f_{l}$ and  $f_{r}$ are of convex type,

\noindent
Case C: $f_{l}$ is of concave type and $f_{r}$ is of convex type, 

\noindent
Case D: both $f_{l}$ and $f_{r}$ are of concave type.
\medskip

Note that by Definition \ref{tip}, convex (concave) type fluxes need not to be convex (concave) functions, see Fig. \ref{noRH}. 
The analysis of the solution of Eq. \eqref{mishrinaJednacina} is very 
similar in all the cases A-D, so we consider only the flux geometry in Case A: \medskip

\noindent
\textbf{Assumption 1:} $f_{l}$ is of convex type, $f_{r}$ is of concave type 
and $f_l(u)>f_{r}(u)$ for all $u\in\Omega$.
\medskip

Let us denote a unique minimum of $f_{l}$ with $\theta_{l}$ and a 
maximum of $f_{r}$ with $\theta_{r}$, i.e.\
$f_{l}(\theta_{l})=\min_{u\in\Omega}f_{l}(u)$ and $f_{r}(\theta_{r})=\max_{u\in\Omega}f_{r}(u)$.
We say that 
$f(u)\sim u^{\nu}$ if there exists a constant $c\in(0,+\infty)$  
such that $\lim_{u \rightarrow\infty}\frac{f(u)}{u^{\nu}}=c$. 

Let $f_{1,\varepsilon}, f_{2,\varepsilon} $ be nets of locally integrable 
functions over some domain $\omega\subset\mathbb{R}\times\mathbb{R}_{+}$, and 
$\mathcal{D}^{\prime}(\omega)$ the space of distributions. 
The relation $\approx$ is defined by
$f_{1,\varepsilon}\approx f_1\in\mathcal{D}^{\prime}(\omega)$ if
$\int_{\omega}f_{1,\varepsilon}\varphi\rightarrow \langle 
f_{1}\delta,\varphi\rangle$
as $\varepsilon\rightarrow 0$ for every test function $\varphi\in 
C_{0}^{\infty}(\omega)$. Here $f_{1,\varepsilon}\approx f_{2,\varepsilon}$ 
means $f_{1,\varepsilon}-f_{2,\varepsilon}\approx 0$ and we say that 
$f_{1,\varepsilon}$ and $f_{2,\varepsilon}$ are distributionally equal (have 
the same distributional limit).

As noted in the introduction, a solution of Eq. 
\eqref{mishrinaJednacina} containing stationary singular 
wave which fits the Rankine-Hugoniot deficit at $x=0$
is expected to appear. We will 
represent singular solutions with shadow waves, nets of piecewise constant 
functions defined as follows.
\begin{df}(\cite{Marko})\label{minor}
A shadow wave is a net of functions 
\begin{equation}\label{sdwdef}
u_{\varepsilon}(x,t)=\left\{
\begin{array}{l}
u_{0}\; ,\quad x<(c-\varepsilon)t\\
u_{0,\varepsilon} \;,\quad (c-\varepsilon)t<x< ct\\
u_{1,\varepsilon} \;,\quad ct<x<(c+\varepsilon)t\\
u_{1} \;, \quad x>(c+\varepsilon)t
\end{array}
\right.,
\end{equation}
where
\begin{equation}\label{alfa i beta}
u_{0,\varepsilon}\sim \varepsilon^{-\alpha},\;u_{1,\varepsilon}\sim\varepsilon^{-\beta}, 
\end{equation}
and $\alpha,\beta\geq 0$. We call $u_{0,\varepsilon}$ and $u_{1,\varepsilon}$ 
intermediate states, and $c$ presents the speed of the shadow wave.\end{df}  
A speed of a stationary shadow wave equals zero, i.e.\ 
$c=0$ in \eqref{sdwdef}.
The value
$k_{\varepsilon}(t) =\varepsilon u_{0,\varepsilon}t + \varepsilon 
u_{1,\varepsilon}t$
is called strength of the shadow wave. A shadow wave is called weakly unique 
if all solutions have the same distributional limit. Here, uniqueness means
that $\lim _{\varepsilon\rightarrow 0}k_\varepsilon(t)=k(t)$
is unique. A limit has the form
\begin{eqnarray*}
u\left( x,t\right) &\approx &\left\{ 
\begin{array}{c}
u_{0},\ x<0 \\ 
u_{1,}\ x>0
\end{array}
\right. + k (t) \delta \left( x\right) , 
\end{eqnarray*} 
where $\delta(x)$ is the Dirac delta function.
The definition bellow defines when a singular solution 
obtained as a limit of a shadow wave as $\varepsilon\rightarrow 0$ is a delta 
shock or a singular shock wave.
\begin{df}\label{minor2}(\cite{Marko}) If $\alpha=1$ or $\beta=1$ in \eqref{alfa i beta}, then $u_{0,\varepsilon},u_{1,\varepsilon}$ are ``major'' 
or $\varepsilon^{-1}$ components. Otherwise, if either $0<\alpha<1$ or 
$0<\beta<1$ then the corresponding component $ u_0,\varepsilon $ or $u_{1} 
\varepsilon$  is called a ``minor component''. 

A delta shock wave is a shadow wave associated with a $\delta$ distribution 
with all minor components having finite limits as $\varepsilon\rightarrow 0$.

A singular shock wave has the feature that either $\left | 
u_{0,\varepsilon}\right | \rightarrow\infty$ and $\varepsilon\left | 
u_{0,\varepsilon}   \right |\rightarrow 0$ or $\left | u_{1,\varepsilon}\right 
| \rightarrow\infty$ and $\varepsilon\left | u_{1,\varepsilon}   \right 
|\rightarrow 0$ when $\varepsilon\rightarrow 0$.
\end{df}

\section{\textbf{Shadow Waves for the Two-Flux Equation}}

We start by considering the overcompressive case, i.e. when \eqref{over} is 
satisfied together with Assumption 1. Suppose that 
all assumptions from Section \ref{ini} are satisfied.
For a solution, we look at the set of all elementary waves plus a shadow wave.

\subsection{Existence of shadow waves in the overcompressive 
case}\label{egzistencija}

Let $f$ from Eq. \eqref{mishrinaJednacina} 
with the initial data \eqref{kx} satisfy
condition \eqref{over}. Then, there is a unique solution 
in the form of a stationary shadow wave. 
\begin{thm}
\label{lema1}
Let
\begin{equation}
\label{let}
f_l(u)\sim u^{\nu_{1}},\; f_r(u)\sim \chi u^{\nu_{2}},
\text{ as } u\to \infty,
\end{equation} 
where $\nu_1, \nu_2\geq 0$,  $\chi\in\{-1,1\}$ and 
\begin{equation}
\label{subquadratic}
\begin{split}
\quad \min(\nu_1,\nu_2)<1 & \quad\text{if }\quad \chi=1, \\
\text{and } \;\nu_1=\nu_2\in[0,2) &\quad \text{if }\quad \chi=-1.
\end{split}
\end{equation} 
Equation \eqref{mishrinaJednacina} together with the initial data of the form 
\eqref{kx} satisfying \eqref{over} has a shadow wave as a 
solution which tends to a singular wave featuring the Dirac $\delta$ function 
\begin{eqnarray}
u\left( x,t\right) &\approx &\left\{ 
\begin{array}{c}
u_{0},\ x<0 \\ 
u_{1,}\ x>0
\end{array}
\right. +\kappa t\delta \left(
x\right) ,  \label{konacno resenje} \end{eqnarray}
where $\kappa$ is given by \eqref{deficit}.
\end{thm}
\begin{proof}
Substitute
\begin{equation}\label{sdw1}
u_{\varepsilon}(x,t)=\left\{
\begin{array}{l}
u_{0}\; ,\quad x<-\varepsilon t\\
u_{0,\varepsilon} \;,\quad -\varepsilon t<x< 0\\
u_{1,\varepsilon} \;,\quad 0<x<\varepsilon t\\
u_{1} \;, \quad x>\varepsilon t
\end{array}
\right.
\end{equation}
with
\begin{equation*}\label{components}
u_{0,\varepsilon}=u_{0}+\xi_{0}\varepsilon^{-\alpha}\;\text{and}\;
u_{1,\varepsilon}=u_{1}+\xi_{1}\varepsilon^{-\beta}\;
\end{equation*}
 and $\alpha,\beta\geq 0$, $0\leq \xi_{0},\xi_1<\infty$ in \eqref{mishrinaJednacina}.\\
From \eqref{sdw1} and \eqref{fluks}, using the standard Rankine-Hugoniot shock calculations we obtain
\begin{align*}\label{ut}
\langle \partial_{t}u_{\varepsilon},\varphi(x,t)\rangle\approx 
\varepsilon\int_{0}^{\infty}(u_{0,\varepsilon}-u_{0})\varphi(-\varepsilon t,t)dt
-\varepsilon\int_{0}^{\infty}(u_{1}-u_{1,\varepsilon})\varphi(\varepsilon t,t)dt.
\end{align*}
Using the Taylor expansion 
\begin{equation}\label{Taylor}
\varphi(\pm \varepsilon t,t)=\varphi(0,t)+\sum_{j=1}^{m}\partial_{x}^{j}\varphi(0,t)\dfrac{(\pm\varepsilon t)^j}{j!}+\mathcal{O}(\varepsilon^{m+1}),
\end{equation}
for $m=1$, we obtain 
\begin{eqnarray*}
\langle \partial_{t}u_{\varepsilon},\varphi(x,t)\rangle &\approx &
\int_{0}^{\infty}\left(-\varepsilon(u_{0}+u_{1})+\varepsilon(u_{0,\varepsilon}+u_{1,\varepsilon} \right)\varphi(0,t)dxdt\\
&&+\int_{0}^{\infty}\varepsilon^2\left(u_{1}-u_{0}+u_{0,\varepsilon}-u_{1,\varepsilon}\right)t\varphi_{x}(0,t)dxdt
\\
&\approx &\langle\left(-\varepsilon(u_{0}+u_{1})+\varepsilon(u_{0,\varepsilon}+u_{1,\varepsilon} \right))\delta(x),\varphi(x,t)\rangle\\
&&-\langle\varepsilon^2\left(u_{1}-u_{0}+u_{0,\varepsilon}-u_{1,\varepsilon}\right)t\delta^{\prime}(x),\varphi(x,t)\rangle.
\end{eqnarray*}
After neglecting all the terms tending to zero, we have
\begin{align}\label{part1}
\langle \partial_{t}u_{\varepsilon},\varphi(x,t)\rangle\approx &\left(\varepsilon^{1-\alpha}\xi_0+\varepsilon^{1-\beta}\xi_1\right)\langle
\delta(x),\varphi(x,t)\rangle
\\&-\left( \varepsilon^{2-\alpha}\xi_0-\varepsilon^{2-\beta}\xi_1\right)\langle
t\delta^{\prime  }(x),\varphi(x,t)\rangle.\nonumber
\end{align}
Next,
\begin{align*}
& \langle  \partial_{x}f(u_{\varepsilon}),\varphi(x,t)\rangle \approx 
\int_{0}^{\infty}(f_{l}(u_{0,\varepsilon})-f_{l}(u_{0}))
\varphi(-\varepsilon t,t)dt \\
+ & \int_{0}^{\infty}\left(f_{r}(u_{1,\varepsilon})-f_{l}
(u_{0,\varepsilon})\right)\varphi(0,t)dt+\int_{0}^{\infty}\left(f_{r}(u_{1})
-f_{r}(u_{1,\varepsilon})\right)\varphi(\epsilon t,t)dt.
\end{align*}

We have to use the Taylor expansion \eqref{Taylor} for $m=2$ now 
since there is a term with $\epsilon^{2}$ above. Using \eqref{deficit} 
we have 
\begin{align}\label{part2}
\langle  \partial_{x}f(u_{\varepsilon}),\varphi(x,t)\rangle & \approx &
-\int_{0}^{\infty}\kappa\varphi(0,t)dt-\varepsilon\int_{0}^{\infty}(f_{l}(u_{0,\varepsilon})+f_{r}(u_{1,\varepsilon}))t\varphi_{x}(0,t)dt\nonumber\\&&+\dfrac{\varepsilon^2}{2}\int_{0}^{\infty}(f_{l}(u_{0,\varepsilon})-f_{r}(u_{1,\varepsilon}))t^{2}\varphi_{xx}(0,t)dt
\nonumber\\
&\approx &\langle -\kappa\delta(x),\varphi(x,t)\rangle+\langle\varepsilon t \left(f_{l}(u_{0,\varepsilon})+f_{r}(u_{1,\varepsilon})\right)\delta^{\prime}(x),\varphi(x,t)\rangle
\nonumber\\&&-\langle \dfrac{\varepsilon^2 t^2}{2}  \left(f_{l}(u_{0,\varepsilon})-f_{r}(u_{1,\varepsilon})\right)\delta^{\prime\prime}(x),\varphi(x,t)\rangle.
\end{align}
Note that $\kappa>0$ due to Assumption 1.
Thus, from \eqref{mishrinaJednacina}, \eqref{part1} and \eqref{part2} 
\begin{align*}
& \langle \partial_{t}u_{\varepsilon}+\partial_{x}f(u_{\varepsilon}),\varphi(x,t)\rangle \approx 
\langle(-\kappa+(\xi_{0}\varepsilon^{1-\alpha}+\xi_{1}\varepsilon^{1-\beta}))\delta(x),\varphi(x,t)\rangle 
\\
&+\langle \left(\varepsilon(f_{l}(u_{0,\varepsilon})+f_{r}(u_{1,\varepsilon}))
-\left(\xi_0 \varepsilon^{2-\alpha}-\xi_1\varepsilon^{2-\beta}\right)\right)t\delta^{\prime}(x),\varphi(x,t)\rangle\\
&-\langle \dfrac{\varepsilon^2 t^2}{2}  \left(f_{l}(u_{0,\varepsilon})-f_{r}(u_{1,\varepsilon})\right)\delta^{\prime\prime}(x),\varphi(x,t)\rangle\approx 0.
\end{align*}

These relations are satisfied if the following holds
\begin{align}
\label{1}\lim_{\varepsilon\rightarrow 0}((\xi_{0}\varepsilon^{1-\alpha}+\xi_{1}\varepsilon^{1-\beta}))&=\kappa\\
\label{2}\lim_{\varepsilon\rightarrow 0}\left(\varepsilon(f_{l}(u_{0,\varepsilon})+f_{r}(u_{1,\varepsilon}))
-\left( \xi_0\varepsilon^{2-\alpha}-\xi_1\varepsilon^{2-\beta}\right)\right)&=0\\
\label{3}\lim_{\varepsilon\rightarrow 0}\varepsilon^2\left(f_{l}(u_{0,\varepsilon})-f_{r}(u_{1,\varepsilon})\right)&=0.
\end{align}
Let us consider relation \eqref{1} first. It is clear that $\alpha,\beta\leq 1$ since the left-hand side would be infinite otherwise (note that $\xi_0,\xi_1\geq0$). This implies that the second term on the left-hand side of relation \eqref{2} vanishes, i.e. relation \eqref{2} becomes
\begin{equation*}\label{2*}\lim_{\varepsilon\rightarrow 0}\varepsilon(f_{l}(u_{0,\varepsilon})+f_{r}(u_{1,\varepsilon}))=0
\end{equation*}
and system \eqref{1}-\eqref{3}
becomes
\begin{align*}
\lim_{\varepsilon\rightarrow 0}((\xi_{0}\varepsilon^{1-\alpha}+\xi_{1}\varepsilon^{1-\beta}))&=\kappa\\
\lim_{\varepsilon\rightarrow 0}\varepsilon(f_{l}(u_{0,\varepsilon})+f_{r}(u_{1,\varepsilon}))&=0\\
\lim_{\varepsilon\rightarrow 0}\varepsilon^2\left(f_{l}(u_{0,\varepsilon})-f_{r}(u_{1,\varepsilon})\right)&=0.
\end{align*}
  
It is clear that at least one of the parameters $\alpha$ or $\beta$ should 
be 1 since $\kappa>0$. 
Taking into account \eqref{let}, the above system becomes 
 \begin{align}
\label{1a}\xi_{0}\varepsilon^{1-\alpha}+\xi_{1}\varepsilon^{1-\beta}&\approx\kappa\\
\label{2b}c_{1}\xi_{0}^{\nu_{1}}\varepsilon^{1-\nu_{1}\alpha}+\chi c_{2}\xi_{1}^{\nu_{2}}\varepsilon^{1-\nu_{2}\beta}&\approx 0\\
\label{3c}c_{1}\xi_{0}^{\nu_{1}}\varepsilon^{2-\nu_{1}\alpha}-\chi c_{2}\xi_{1}^{\nu_{2}}\varepsilon^{2-\nu_{2}\beta}&\approx 0,
\end{align}
where $c_1,c_{2}>0$ are constants depending on the behavior of $f_l$ at infinity. 
The fluxes $f_l$ and $f_r$ satisfy \eqref{subquadratic} and $\alpha=1$ or 
$\beta=1$, so one can easily check that the system (\ref{1a} -- \ref{3c}) 
is consistent.

\noindent\textbf{Case 1.} Let us first consider the case $\chi =1$. 
System (\ref{1a}--\ref{3c}) is then 
\begin{align*}
\xi_{0}\varepsilon^{1-\alpha}+\xi_{1}\varepsilon^{1-\beta}&\approx\kappa\\
c_{1}\xi_{0}^{\nu_{1}}\varepsilon^{1-\nu_{1}\alpha}+c_{2}\xi_{1}^{\nu_{2}}\varepsilon^{1-\nu_{2}\beta}&\approx 0\\
c_{1}\xi_{0}^{\nu_{1}}\varepsilon^{2-\nu_{1}\alpha}-c_{2}\xi_{1}^{\nu_{2}}\varepsilon^{2-\nu_{2}\beta}&\approx 0.
\end{align*}
%
{If both $\nu_1<1$ and $\nu_2<1$, then the above conditions are satisfied for $\alpha=\beta=1$ and $\xi_0+\xi_1=\kappa$. On the other hand, if $\nu_1<1$ and $\nu_2\geq 1$, then the conditions are satisfied for $\alpha=1$, $\beta=0$, $\xi_0=\kappa$ and $\xi_1=0$. Similarly, if $\nu_1\geq 1$ and $\nu_2<1$, then (\ref{1a}--\ref{3c}) holds for $\alpha=0$, $\beta=1$, $\xi_0=0$ and $\xi_1=\kappa$. }\\

\par\noindent \textbf{Case 2.} Let us now suppose that $\chi=-1$. System (\ref{1a}--\ref{3c}) then becomes
\begin{align*}
\xi_{0}\varepsilon^{1-\alpha}+\xi_{1}\varepsilon^{1-\beta}&\approx\kappa\\
c_{1}\xi_{0}^{\nu_{1}}\varepsilon^{1-\nu_{1}\alpha}-c_{2}\xi_{1}^{\nu_{2}}\varepsilon^{1-\nu_{2}\beta}&\approx 0\\
c_{1}\xi_{0}^{\nu_{1}}\varepsilon^{2-\nu_{1}\alpha}+c_{2}\xi_{1}^{\nu_{2}}\varepsilon^{2-\nu_{2}\beta}&\approx 0.
\end{align*}

If $\nu_1<1$ and $\nu_2<1$, $\nu_1<1$ and $\nu_2\geq 1$ or $\nu_1\geq 1$ and $\nu_2<1$ we have the same results under the same conditions in as in Case 1. \\
If $\nu_1=1$ and $\nu_2\geq 1$, one can easily see that system (\ref{1a}--\ref{3c}) is satisfied provided $\alpha=1$, $\beta=\dfrac{1}{\nu_2}$, $\xi_0=\kappa$ and $\xi_1=\left( \dfrac{c_1}{c_2}\kappa\right)^\frac{1}{\nu_2}$, whereas for $\nu_1\geq 1$ and $\nu_2=1$ we obtain $\alpha=\dfrac{1}{\nu_1}$, $\beta=1$, $\xi_0=\left( \dfrac{c_2}{c_1}\kappa\right)^\frac{1}{\nu_1}$ and $\xi_1=\kappa$.\\
If $\nu_1=\nu_2\in[1,2)$, for $\alpha=\beta=1$, one gets $\xi_0,\xi_1$ from the system
\begin{align*}
\xi_{0}+\xi_{1}&=\kappa\\
c_{1}\xi_{0}^{\nu_{1}}-c_{2}\xi_{1}^{\nu_{1}}&=0,
\end{align*}
i.e., 
\[\xi_0= \dfrac{\kappa\left(\ \dfrac{c_2}{c_1} \right)^\frac{1}{\nu_1}}{1+\left(\dfrac{c_2}{c_1}\right)^\frac{1}{\nu_1}} \quad , \quad\xi_1=\dfrac{\kappa}{1+\left(\dfrac{c_2}{c_1}\right)^\frac{1}{\nu_1}}.\] \\
On the other hand, if  $\nu_1\neq\nu_2\in [1,2)$, we have two cases to consider: $\nu_1<\nu_2$ and $\nu_1>\nu_2$. Let $\nu_1<\nu_2$. Then system (\ref{1a}--\ref{3c}) is satisfied for $\alpha=1$, $\beta=\dfrac{\nu_1}{\nu_2}$, $\xi_0=\kappa$ and $\xi_1=\bigg(\dfrac{c_1}{c_2}  \kappa^{\nu_{1}} \bigg)^\frac{1}{\nu_2}.$ Similarly, if $\nu_1>\nu_2$ one gets that the conditions are satisfied for $\alpha=\dfrac{\nu_2}{\nu_1}$, $\beta=1$, $\xi_0=\bigg( \dfrac{c_2}{c_1}\kappa^{\nu_2}\bigg)^\frac{1}{\nu_1}$ and $\xi_1=\kappa$.$\quad\quad\quad\quad\quad\quad\square$
\end{proof}

The limits of obtained shadow waves as $\varepsilon \rightarrow 0$ 
are given in Table \ref{ntab} according to Definition \ref{minor2}.



\begin{table}[]
\begin{tabular}{@{}p{25mm}p{60mm}p{30mm}@{}}
\toprule
$\mathbf{\nu_1, \nu_2}$&\textbf{Solution of system (\ref{1a}--\ref{3c})}&\textbf{Shadow wave when $\mathbf{\varepsilon\rightarrow 0}$}\\\midrule
\multicolumn{3}{c}{\textbf{Case 1:} $\mathbf{\chi=1}$}    \\ \midrule
$\nu_1<1, \nu_2<1$        & $\alpha=\beta=1\;$, $\;\xi_1+\xi_2=\kappa$        & delta shock wave      \\
$\nu_1<1, \nu_2\geq 1$        & $\alpha=1\;, \;\beta=0,\; \xi_0=\kappa, \;\xi_1=0$        & delta shock wave         \\
$\nu_1\geq 1, \nu_2<1$       & $\alpha=0\;,\; \beta=1, \;\xi_0=0, \;\xi_1=\kappa$       & delta shock wave         \\ \midrule
\multicolumn{3}{c}{\textbf{Case 2: $\mathbf{\chi=-1}$}} \\ \midrule
$\nu_1<1, \nu_2<1$        & $\alpha=\beta=1\;$, $\;\xi_1+\xi_2=\kappa$        & delta shock wave      \\
$\nu_1<1, \nu_2\geq 1$        & $\alpha=1\;, \;\beta=0\;,\; \xi_0=\kappa\;,\; \xi_1=0$        & delta shock wave         \\
$\nu_1\geq 1, \nu_2<1$       & $\alpha=0\;,\; \beta=1\;, \;\xi_0=0\;,\; \xi_1=\kappa$       & delta shock wave         \\ 
$\nu_1=1, \nu_2\geq 1$        & $\alpha=1, \beta=\dfrac{1}{\nu_2}, \xi_0=\kappa, \xi_1=\bigg( \dfrac{c_1}{c_2}\kappa\bigg)^\frac{1}{\nu_2}$        & singular shock wave      \\
$\nu_1\geq 1, \nu_2=1$        & $\alpha=\dfrac{1}{\nu_1}, \beta=1, \xi_0=\bigg( \dfrac{c_2}{c_1}\kappa\bigg)^\frac{1}{\nu_1}, \xi_1=\kappa$        & singular shock wave       \\
$\nu_1=\nu_2\in[1,2)$        &  $\alpha=\beta=1\;,$ \newline $\xi_0=\dfrac{\kappa \bigg(\dfrac{c_2}{c_1}\bigg)^\frac{1}{\nu_1}}{\bigg(1+\dfrac{c_2}{c_1}\bigg)^\frac{1}{\nu_1}}\; , \; \xi_1= \dfrac{\kappa}{\bigg(1+\dfrac{c_2}{c_1}\bigg)^\frac{1}{\nu_1}}$        & delta shock wave       \\
$\nu_1,\nu_2\in[1,2),$\newline $\nu_1<\nu_2$        & $\alpha=1\;$, $\;\beta=\dfrac{\nu_1}{\nu_2}\;$, \newline$\xi_0=\kappa\;,$ $\;\xi_1=\bigg(\dfrac{c_1}{c_2}\kappa^{\nu_1}\bigg)^{\frac{1}{\nu_2}}$     & singular shock wave        \\
$\nu_1,\nu_2\in[1,2),$\newline$ \nu_1>\nu_2$         & $\alpha=\dfrac{\nu_2}{\nu_1}\;,\;\beta=1\;,$\newline
$\xi_0=\bigg(\dfrac{c_2}{c_1}\kappa^{\nu_2}\bigg)^{\frac{1}{\nu_1}}\;,\;\xi_1=\kappa$        & singular shock wave       \\ \bottomrule
\end{tabular}
\caption{The limits of the shadow wave solution in various cases when $\varepsilon\rightarrow 0$}
\label{ntab}
\end{table}

\begin{ex} \label{pp1}
Let
\begin{equation}\label{konkretno}
f_l(u)=\dfrac{1+2u^{2}}{1+u^2}\;, \; f_r(u)=-\dfrac{1+2u^{2}}{1+u^2},
\end{equation}
and $u_0(x)=1$.

One can easyly check that $f_l^{\prime}(1)> 0> f_r^{\prime}(1)$, i.e. that the overcompressibility condition is satisfied. Using that
$\kappa=f_l(1)-f_r(1)=3$, $\chi=-1$, $\nu_1<1$ and $\nu_2<1$ 
and Theorem \ref{lema1}, 
we obtain a solution in the form of a stationary shadow wave 
(i.e. delta shock wave, see Table \ref{ntab}) with the strength equal 
to $k_\varepsilon(t)=3t$, 
\begin{equation*}
u(x,t)\approx 1+3t\delta(x).
\end{equation*}

\end{ex}

\begin{ex}
\label{pp2}
Let 
\[f_l(u)=\sqrt{u^2+1}+1\;,\; f_r(u)=\dfrac{1}{\sqrt{u^2+1}},\]
and $u_0(x)=1.$

In this case we have $\nu_1\geq 1, \nu_2<1$ and $\chi=1$. Thus we obtain a delta shock wave of the form
\begin{equation*}
u(x,t)\approx 1+\dfrac{2+\sqrt{2}}{2}t\delta(x).
\end{equation*}
\end{ex}

\subsection{Solutions to Riemann problem}\label{riman}

Let us now consider problem \eqref{mishrinaJednacina} with a
general Riemann initial data
\begin{equation*}
\label{kxOpsti1}
u(x,0)=\left\{\begin{array}{ll}
	u_{l},& x<0\\
	u_{r},& x>0
	\end{array}\right..
\end{equation*}

There are the four possible positions of the initial data states $u_l,\;u_r$ with respect to $\theta_l,\;\theta_r$: 

\noindent
i) $u_l\geq \theta_l,\;u_r\geq\theta_r$

\noindent
ii)]$u_l\geq \theta_l,\;u_r<\theta_r$

\noindent
iii) $u_l<\theta_l,\;u_r\geq\theta_r$

\noindent
iv) $u_l<\theta_l,\;u_r<\theta_r$ 

\noindent
due to Assumption 1. One can see the illustration at Fig. \ref{ilustracije}.
\begin{figure}[h]
\begin{center}
\begin{tabular}[ht]{@{}cc@{}}
\small{i) $u_l\geq \theta_l,\;u_r\geq\theta_r$}& \small{ii) $u_l\geq \theta_l,\;u_r<\theta_r$ }\\
\includegraphics[height=40mm]{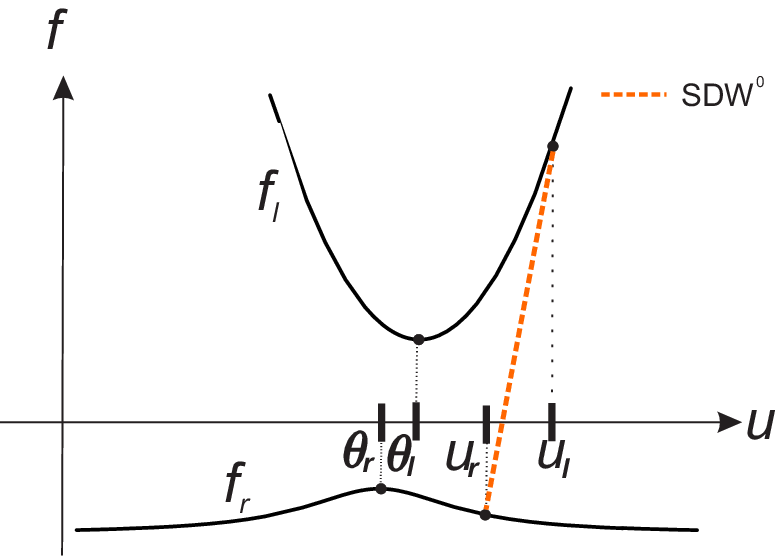}& \includegraphics[height=40mm]{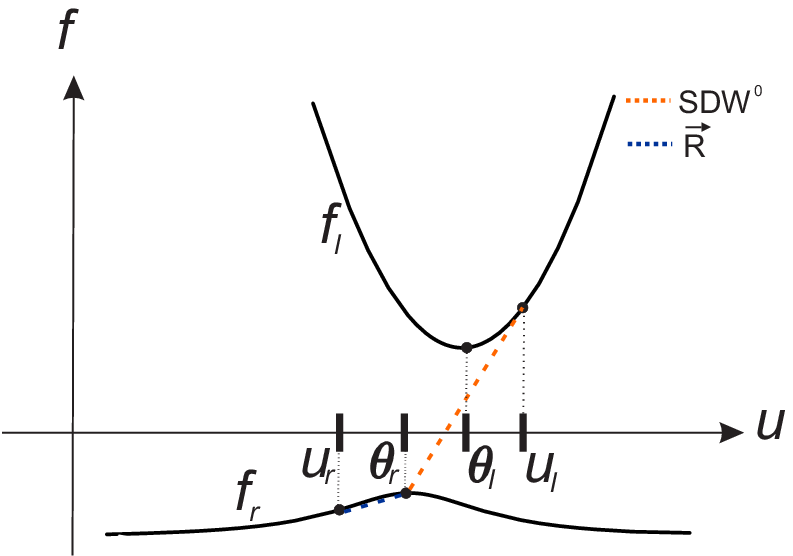}\\
\small{iii) $u_l<\theta_l,\;u_r\geq\theta_r$}& \small{iv) $u_l<\theta_l,\;u_r<\theta_r$ }\\
\includegraphics[height=40mm]{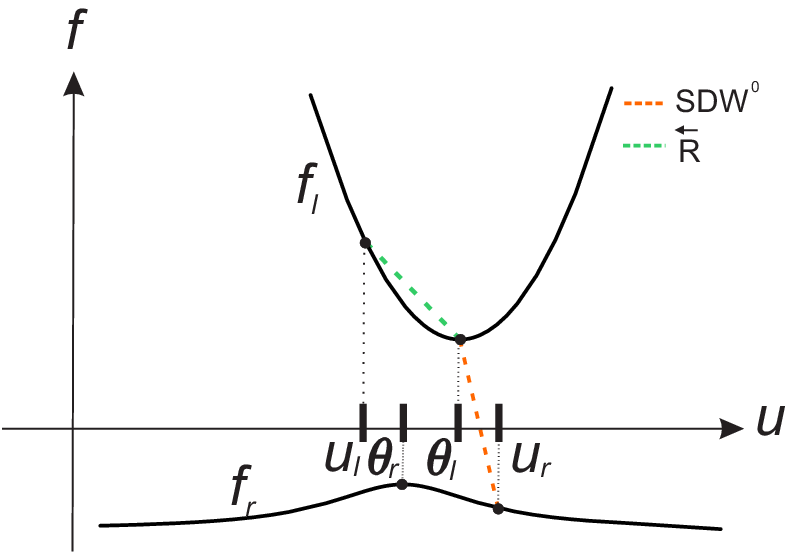} & \includegraphics[height=40mm]{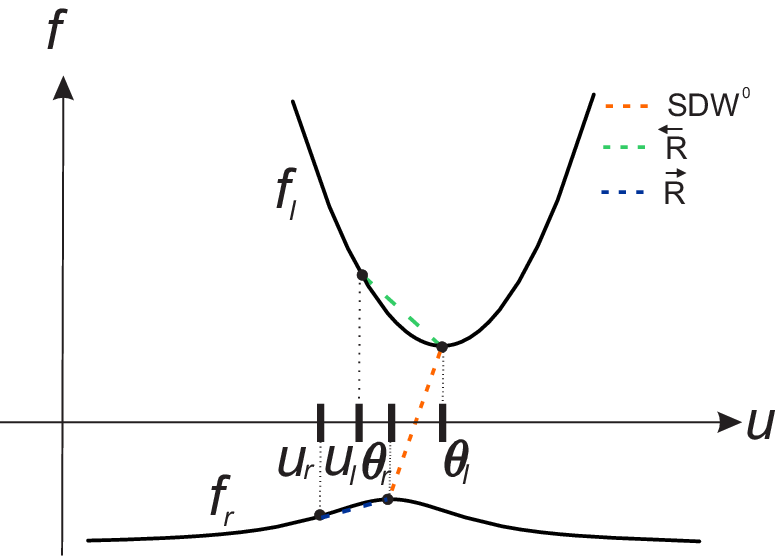}
\end{tabular}
\end{center}
\caption{Various positions of initial data states $u_l,u_r$ with respect to $\theta_l,\;\theta_r$}
\label{ilustracije}
\end{figure}

In case i), the overcompressibility condition \eqref{over} is satisfied and 
there exists a stationary shadow wave at $x=0$ with the sates $u_l,u_r$. We 
will denote it by $SDW^0(u_l,u_r)$. 
Note that in this case the shadow wave is uniquely determined. Namely, if 
there would exist a state $\hat{u}_r>\theta_r$ and $\hat{u}_r\neq u_r$ that 
determines a stationary shadow wave at $x=0$, i.e. $SDW^0(u_l,\hat{u}_r)$, 
then the pair $(\hat{u}_r,u_r)$ would determine a wave of the right-hand side 
of $x=0$ propagating to left since $f_r^{\prime}(\hat{u}), 
f_r^{\prime}(\hat{u_r})\leq 0$, but that is impossible. 

In case ii), we have $f_l^{\prime}(u_l)>0$, $f_r^{\prime}(u_r)>0$, i.e. the 
overcompressibility condition is violated. Thus, a state $u_r$ has to be 
connected with a state $u_1$, such that $f_r^{\prime}(u_1)\leq 0$, i.e. 
$u_1\geq \theta_r$.
Then the pair $u_l,u_1$ satisfies the overcompressibility condition and 
Theorem \ref{lema1} implies the existence of a $SDW^0(u_l,u_1)$. The states 
$u_l$ and $u_1$ are then connected with a forward wave that is in fact a 
rarefaction wave since $u_1>u_r$ and $f_r$ is of concave type.  
The unique choice of the state $u_1$ is $u_1=\theta_r$ since for any 
$u_1>\theta_r$ we would have 
$f_{r}^{\prime}(u_r)>0$,$f_{r}^{\prime}(u_1)<0$,  
but that is not allowed, too. 
Thus, a unique solution is of the form 
$SDW^0(u_l,\theta_r)+\overrightarrow{R}(\theta_r,u_{r})$
in this case. Here $\overrightarrow{R}(u_{1},u_{r})$ 
denotes the forward rarefaction wave and ``+" means ``followed by".

Similarly, the only proper choice is $u_0=\theta_l$ in case iii). It is 
connected to $u_l$ in such a way that the pair 
$\left(u_l,u_0\right)$ satisfies the 
overcompressibility condition. Also, there is no interaction of the backward 
rarefaction  wave $\overleftarrow{R}(u_{l},u_{0})$ with the $SDW^0(u_0,u_r)$.

Case iv) is a combination of the previous cases. A unique solution is a wave 
combination  $\overleftarrow{R}(u_l,\theta_l)+SDW^0(\theta_l,\theta_r)
+\overrightarrow{R}(\theta_r,u_r)$.
\medskip

Table \ref{resenja} shows solutions of the Riemann problem for \eqref{mishrinaJednacina}, \eqref{kxOpsti} in cases i)-iv).

\begin{table}[h]
\centering
\begin{tabular}{@{}cll@{}}
\toprule
 Subcase&Position of $u_l,u_r$& Solution \\
 \midrule
\rule{0pt}{2ex}i)&$u_l\geq\theta_l$ and $u_r\geq\theta_r$& $
SDW^0(u_l,u_r)$  \\
\rule{0pt}{4ex}ii)&$u_l\geq\theta_l$ and $u_r<\theta_r$ &$
SDW^0(u_l,\theta_r)+\overrightarrow{R}(\theta_r,u_r)$ \\
\rule{0pt}{4ex}iii)& $u_l<\theta_l$ and $u_r\geq\theta_r$& $
\overleftarrow{R}(u_l,\theta_l)+SDW^0(\theta_l,u_r)$ \\
\rule{0pt}{4ex}iv)&$u_l<\theta_l$ and $u_r<\theta_r$ &$
\overleftarrow{R}(u_l,\theta_l)+SDW^0(\theta_l,\theta_r)+\overrightarrow{R}(\theta_r,u_r)$\\ \bottomrule
\end{tabular}
\caption{Solutions of problem \eqref{mishrinaJednacina}, \eqref{kxOpsti}}
\label{resenja}
\end{table}

\subsection{Remark on the linear case}\label{prosirenje}
Suppose that both fluxes are affine.  The above method with a centered 
shadow wave can also be applied. As we mentioned in the introduction,  
when both fluxes are linear and 
$f_l^{\prime}>0>f_r^{\prime}$, there exists no classical solution for 
arbitrary initial values $u_l,\;u_r$ since any pair $(u^-,u^+)$ ($u^-$ and 
$u^+$ are given by \eqref{uPlus}) which satisfies the Rankine-Hugoniot jump 
condition at $x=0$ determines a wave connecting $u_l$ and $u^-$ at the 
left-hand side of the origin traveling to right, and a right-hand side wave 
connecting $u_r$ and $u^+$ traveling to left that are not allowed (see 
illustration in Fig. \ref{nova}).
\begin{figure}[h]
\begin{center}
\includegraphics[width=0.4\textwidth]{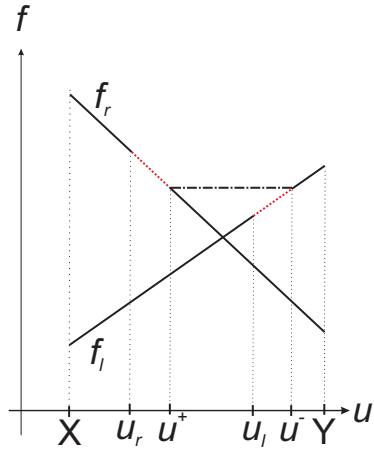}
\end{center}
\caption{The overcompressive linear flux case}
\label{nova}
\end{figure}

Put $f_l(u)=p_{l}u+q_{l}$, $f_r(u)=p_{r}u+q_{r}$ in Eq.
\eqref{mishrinaJednacina} where
$p_l\geq0,\;p_r\leq0$ and use the initial data \eqref{kx}. 
It is clear that
\begin{equation*}
f_l^{\prime}(u_0)> 0>f_r^{\prime}(u_1)
\end{equation*}
for every choice of the states $u_0,u_1$. 
That fits the case $\chi=1$, $\nu_1=\nu_2\in [1,2)$ 
(see Table \ref{ntab}), i.e. $\alpha =1$, $\;\beta=1$, 
$\xi_0=\dfrac{\kappa\left(\frac{c2}{c1}\right)}{1+\frac{c2}{c1}}$ and 
$\xi_1=\dfrac{\kappa}{1+\left(\frac{c2}{c1}\right)}.$

On the other hand, the Rankine-Huginot deficit equals
\begin{equation*}
   \kappa=f_l(u_0)-f_r(u_1)=p_lu_0+q_l-p_ru_1-q_r.
\end{equation*}
According to Theorem \ref{lema1}, the problem has a solution in the form of a $SDW^0(u_0,u_1)$ tending to a delta shock wave as $\varepsilon \rightarrow 0$, i.e.
\begin{equation}\label{deltaLin}
u(x,t)\approx U+\left(p_lu_0-p_ru_1+q_l-q_r\right)t\delta(x),
\end{equation}
where we denoted \eqref{kx} with $U(x)$, i.e.
\begin{equation*}\label{U}
U(x)=\left\{\begin{array}{ll}
	u_{0},& x<0\\
	u_{1},& x>0
	\end{array}\right..
\end{equation*}

\begin{ex}
Let us consider Eq. \eqref{mishrinaJednacina} where the fluxes are given by
\begin{equation}\label{line}
f_l=u\quad\text{and}\quad f_r=-u
\end{equation}
together with the constant initial data $u(x,0)=1.$
\label{primer2}
\end{ex}
\par According to \eqref{deltaLin}, problem   
\eqref{mishrinaJednacina}, \eqref{line} has a solution in the form $SDW^0(1,1)$ with the strength $\kappa t=(f_l(1)-f_r(1))t=2t$. As we can see from Table \ref{ntab}, the obtained shadow wave is a delta shock wave.\\
The case with linear fluxes is convenient for using the delta split solution concept \cite{m15,m18}, so one could derive a delta shock wave solution on an alternative way and compare it with the solution obtained using shadow waves.

\subsection{A numerical example}\label{numerika111}

We will use the well known Godunov method for the numerical example bellow. Take		
\begin{align}\label{nonlin}
v_{t}+F(v)_x=0,
\end{align}
with
$v=(v^1,v^2)$ being a two component vector with components 
$v^i:\mathbb{R}\times\mathbb{R}_+\rightarrow \mathbb{R},\;i=1,2$ and 
$F:\mathbb{R}^2\rightarrow\mathbb{R}^2$,
be a nonlinear system of conservation laws together with the initial data
\begin{equation*}
v(x,0)=\left\{\begin{array}{ll}\label{vkx}
	v_l,& x<0\\
	v_r,& x>0
	\end{array}\right.,
\end{equation*}
where $v_l=(v_{l}^1,v_{l}^2)$ and $v_r=(v_{r}^1,v_{r}^2)$.
Solving system \eqref{nonlin} numerically starts by discretizing the $x-t$ 
plane using a uniform grid $\Delta x\mathbb{Z}\times\Delta 
t\mathbb{N},\;\Delta x>0,\Delta t>0$ with grid points labeled 
$(x_j,t_n):=(j\Delta x,n\Delta t)$, the discrete unknowns $u_j^n=u(x_j,t_n)$, 
and using a suitable numerical scheme subsequently. Herein, we use Godunov's 
method coupled with Roe's linearization (see \cite{LeVeque}). The basic 
principle of the method can be described as follows.
First, system \eqref{nonlin} is linearized by approximating system \eqref{nonlin} locally at each cell interface
\begin{equation}\label{ro}
v_t+\widehat{A}_{j-1/2}v_x=0,
\end{equation}
where $\widehat{A}_{j-1/2}$ is a $2\times 2$ matrix that
approximate $F^{\prime}(v)$ in the neighborhood of $v_{j-1}$ and $v_j$ 
having the following properties:
\begin{align}
\widehat{A}_{j-1/2}(v_{j-1},v_j)(v_{j}-v_{j-1})=F(v_j)-F(v_{j-1})&\label{r1}\\
\widehat{A}_{j-1/2}(v_{j-1},v_{j})\quad\text{ has real distinct eigenvalues}&\label{r2}\\
\widehat{A}_{j-1/2}(v_{j-1},v_j)\rightarrow F^{\prime}(\bar{v})\quad\text{when}\quad v_{j-1},v_{j}\rightarrow \bar{v}.&\label{r3}
\end{align}
Note that condition \eqref{r2} implies that $\widehat{A}_{j-1/2}$ 
is diagonizable, and it can be decomposed
\begin{equation*}
\widehat{A}_{j-1/2}=R_{j-1/2}\Lambda_{j-1/2} R_{j-1/2}^{-1},
\end{equation*}
with $\Lambda_{j-1/2} =\mathop{\rm diag}
(\lambda_{j-1/2}^{1},\lambda_{j-1/2}^{2})$ 
being a diagonal matrix of eigenvalues and $R_{j-1/2}=[r_{j-1/2}^{1}\mid 
r_{j-1/2}^{2}]$ being the matrix of appropriate eigenvectors. Let us
introduce the following notation 
\begin{equation*}
\begin{split}
\lambda _{j-1/2,i}^{+}=& \max (\lambda _{j-1/2,i},0),\; 
\Lambda_{j-1/2} ^{+}=\mathop{\rm diag}\left(
\lambda _{j-1/2,1}^{+},\lambda _{j-1/2,2}^{+}\right), \\
\lambda _{j-1/2,i}^{-}=& \min (\lambda _{j-1/2}^{i},0),\; 
\Lambda_{j-1/2} ^{-}=\mathop{\rm diag}\left(
\lambda _{j-1/2,1}^{-},\lambda _{j-1/2,2}^{-}\right), \\
\widehat{A}_{j-1/2}^{+} =& R_{j-1/2}\Lambda_{j-1/2} ^{+}R_{j-1/2}^{-1},\; 
\widehat{A}_{j-1/2}^{-}=R_{j-1/2}\Lambda_{j-1/2} ^{-}R_{j-1/2}^{-1},\;
i=1,2.
\end{split}
\end{equation*}
Now, for the linearized system (\ref{ro}) Godunov method takes the form 
\begin{equation}
v_{j}^{n+1}=v_{j}^{n}-\frac{\Delta t_{n}}{\Delta x_{j}}
\left[ \widehat{A}_{j+1/2}^{-}(v_{j+1}^{n}-v_{j}^{n})
+\widehat{A}_{j-1/2}^{+}(v_{j}^{n}-v_{j-1}^{n})\right] .
\label{godunov}
\end{equation}
\bigskip

Let us verify the results obtained in  Example \ref{pp1} numerically. 
Suppose that \eqref{mishrinaJednacina} is the 
nonlinear system of conservation laws 
\begin{align}\label{sistemH}
u_t+(hf_r(u)+(1-h)f_l(u))_x&=0\\
h_t&=0,\nonumber
\end{align}
with the initial data
\begin{equation*}
u(x,0)=1,\quad h(x,0)=\left\{\begin{array}{ll}
	0,& x<0\\
	1,& x>0
	\end{array}\right..
\end{equation*}
The functions $f_l,f_r$ are given by \eqref{konkretno}. Using the Roe linearization, system 
\eqref{sistemH} is approximated 
by the linear system \eqref{ro} at every cell interface where
\begin{equation*}
\widehat{A}_{j-1/2}=\left[\begin{array}{ll}
\gamma_{j-1/2}&\rho_{j-1/2}\\
&\\
 0& 0
\end{array}
\right],
\end{equation*}
and
\begin{equation*}
\gamma_{j-1/2}=\dfrac{(1-2h_{j-1})(u_{j-1}+u_j)}{(1+u_{j-1}^2)(1+u_{j}^2)}\quad\text{ and}\quad \rho_{j-1/2}=\dfrac{-2(1+2u_{j}^2)}{(1+u_{j}^2)}.
\end{equation*}
The eigenvalues of $\widehat{A}_{j-1/2}$ are 
$\lambda_{j-1/2,1}=\gamma_{j-1/2}$ and $\lambda_{j-1/2,2}=0$. 
One can easily check that conditions \eqref{r1} and \eqref{r3} are satisfied. 
Condition \eqref{r2} also holds as $\gamma_{j-1/2}\neq 0$ for all $j$ since 
$u_l+u_r=2>0$ and the expected solution is a shadow wave above the line $u=1$. 
This implies $u_{j-1}+u_j\neq 0$ for all $j$. \\ Now, we use Godunov's wave 
propagation method for linear systems of hyperbolic conservation laws 
\eqref{godunov} with a mesh size $\Delta x=0.01$ and obtain results shown in 
Fig.\ref{numerika}. 

\begin{figure}[ht]
\begin{center}
\begin{tabular}[ht]{@{}cc@{}}
\includegraphics[height=35mm]{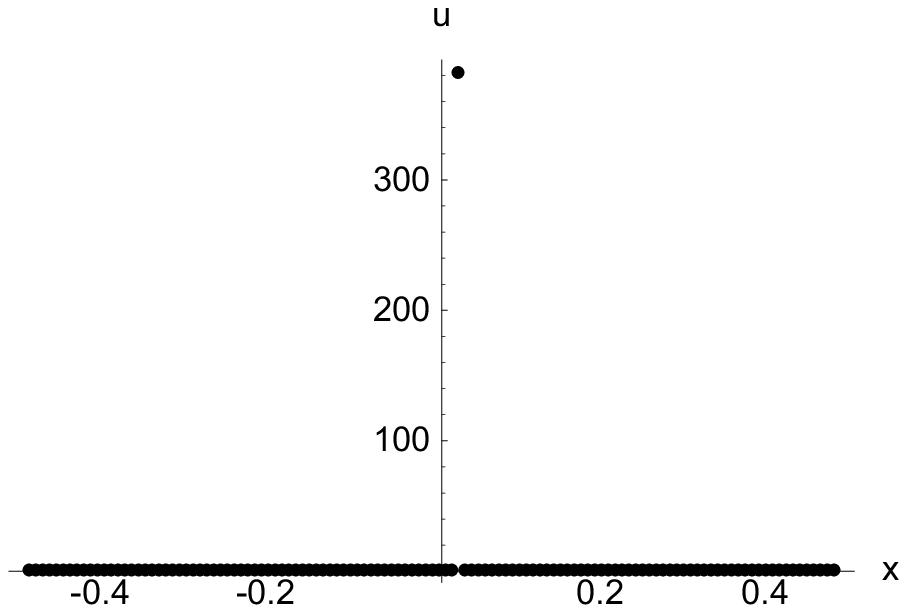} & \includegraphics[height=35mm]{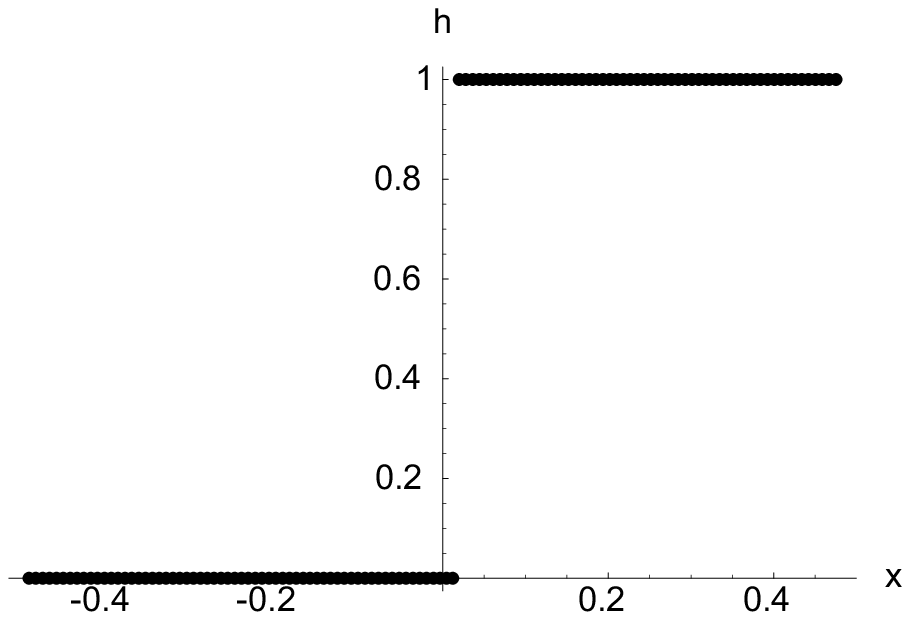} \\
\small{A) $u(x) \;\text{at}\; t=1$}& \small{B) $h(x)\;\text{ at}\; t=1$ }\\
\includegraphics[height=35mm]{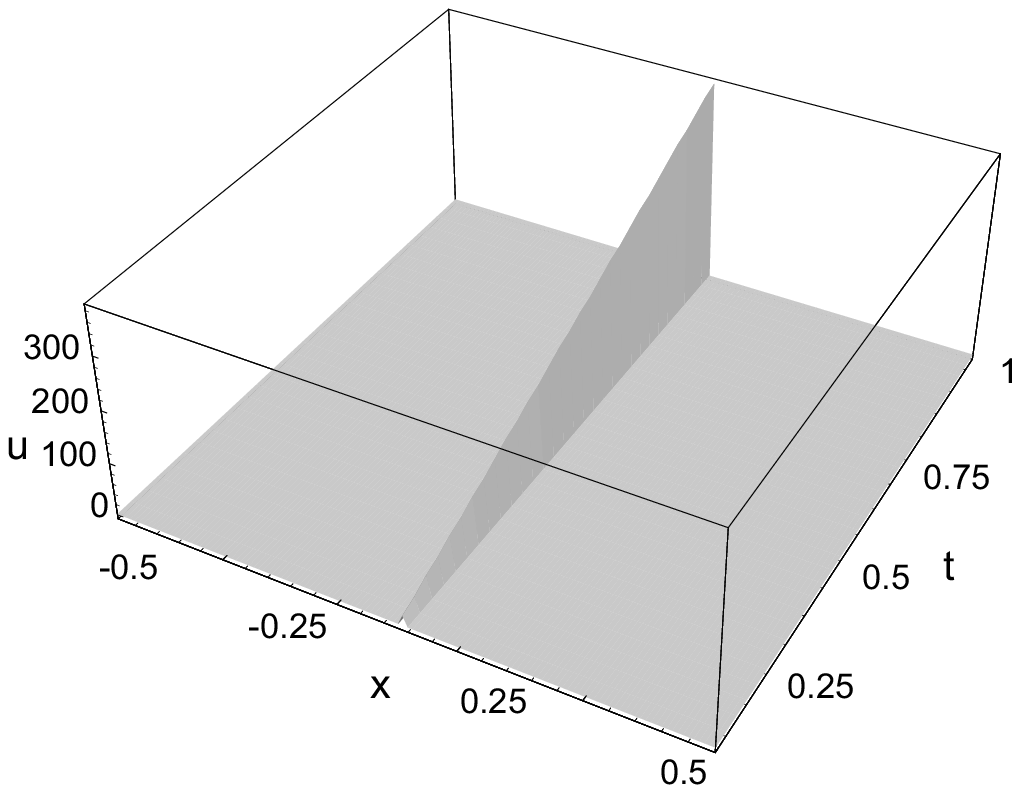} & \includegraphics[height=35mm]{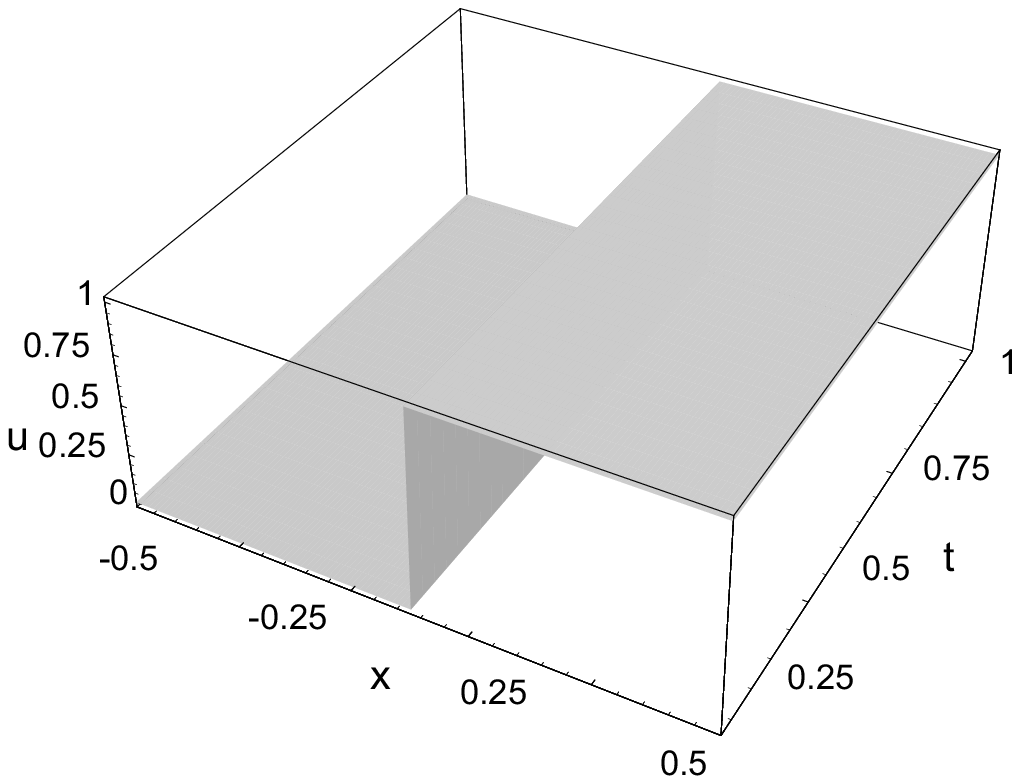} \\
\small{C) $u(x)\;\text{ for}\; t\in[0,1]$}& \small{D) $h(x) \;\text{for}\; t\in[0,1]$ }\\
\end{tabular}
\end{center}
\caption{Numerical results for the problem in Ex. \ref{pp1} }
\label{numerika}
\end{figure}

Note that the unknown $h$ in system \eqref{sistemH} is just an auxiliary variable, presenting in fact the discrete Heaviside function. 
As expected, $h$ does not change during time. On the other hand, the numerical 
approximation of $u$ is unbounded. In order to verify if it really presents a 
delta shock wave, we perform a test we first introduced in \cite{natasa}. 
Denote the singular part of the expected delta shock wave \eqref{konacno resenje} with $u_s$, i.e.\ $u_s=\kappa t$,
where $\kappa$ is given by \eqref{deficit}. Next, denote 
the surface under a singular part of the solution with
\begin{equation*}
P(t)=\int u_s dx=\kappa t\int \delta(x) dx\approx\kappa t=3t,
\end{equation*}
since it is well known that 
$\int \delta(x)dx\approx 1$,
and  as seen in Example \ref{pp1}, in our case $\kappa=3$.
In our numerical test, we approximate $P(t)$ using the left Riemann sum, 
denoted with $P_l(t)$. As one can see in Fig. \ref{surface}, the true 
(dashed line) and calculated (full line) surface under the singular part of 
the delta shock wave are very close over the interval $[0,1]$, so we conclude 
that the obtained unbounded numerical result is actually the expected 
$SDW^0(1,1)$. 
\begin{figure}[ht]
\begin{center}
\begin{tabular}[ht]{@{}c@{}}

\includegraphics[height=45mm]{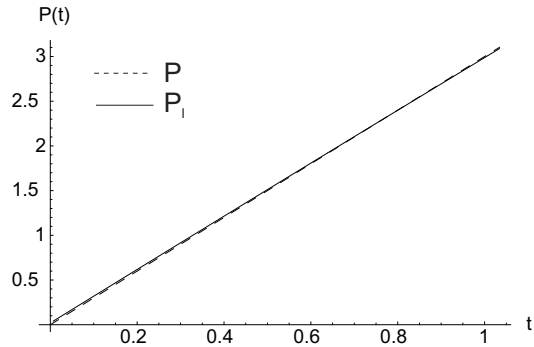}

\end{tabular}
\end{center}
\caption{Delta shock wave verification test for the solution in Ex. \ref{pp1}}
\label{surface}
\end{figure}


Tanja  Kruni\'{c} \\ The Higher Education Technical School of Professional Studies
Novi Sad,  Serbia \\ email: krunic@vtsns.edu.rs
\bigskip 
 
Marko Nedeljkov \\
Department of Mathematics and Informatics, Faculty of Sciences, University od Novi Sad, Serbia \\ email: marko@dmi.uns.ac.rs 

\end{document}